\documentclass[10 pt]{amsart}
\usepackage{mathtools}
\usepackage{amssymb}
\usepackage{amsthm}
\usepackage{amsmath}
\usepackage{cite}
\usepackage{hyperref}
\usepackage{multicol}
\usepackage{array}
\usepackage{enumerate}
\usepackage{graphicx}
\usepackage{calc}
\usepackage{array}
\newcolumntype{P}[1]{>{\centering\arraybackslash}p{#1}}
\newcolumntype{M}[1]{>{\centering\arraybackslash}m{#1}}

\makeatletter
\newcommand*{\DivideLengths}[2]{%
	\strip@pt\dimexpr\number\numexpr\number\dimexpr#1\relax*65536/\number\dimexpr#2\relax\relax sp\relax
}
\makeatother

\title{Bergman-Einstein metrics on two-dimensional Stein spaces}
\author{Soumya Ganguly}
\address{Department of Mathematics, University of California San Diego, 9500 Gilman Dr, La Jolla, CA 92093, USA}
\email{s1gangul@ucsd.edu}
\author{Shubham Sinha}
\address{Department of Mathematics, International Centre for Theoretical Physics, Strada Costiera 11, 34151 Trieste, Italy}
\email{ssinha1@ictp.it}
\date{}
\numberwithin{equation}{section}
\newtheorem{theorem}{Theorem}
\newtheorem*{theorem*}{Theorem}
\newtheorem{proposition}{Proposition}[section]
\newtheorem{remark}{Remark}[section]

\newtheorem{lemma}{Lemma}[section]

\newtheorem*{claim*}{Claim}

\newtheorem{conjecture}{Conjecture}

\DeclareMathOperator{\tr}{\mathrm{tr}}
\begin{document}
	\maketitle
\footnotetext[1]{The first author, Soumya Ganguly, was partially supported by the NSF through the grants DMS-1800549, DMS-1900955 and DMS-2154368. The second author, Shubham Sinha, was supported by the NSF through grant DMS 1802228.}

\footnotetext[2]{MSC classes: 32Q20 (Primary), 32V20, 32C15, 32M18.}

	
	\begin{abstract}
		We show that the Bergman metric of the ball quotients $\mathbb{B}^2/\Gamma$,  where $\Gamma$ is a finite and fixed point free group, is K\"ahler-Einstein if and only if $\Gamma$ is trivial. As a consequence, we characterize the unit ball $\mathbb{B}^2$, among 2 dimensional Stein spaces with isolated normal singularities, proving an algebraic version of Cheng's conjecture for 2 dimensional Stein spaces.  
	\end{abstract}
	
	\section{Introduction}
	The Bergman kernel and metric were introduced by S. Bergman \cite{BergmanpaperonBergmankernel1, BergmanpaperonBergmankernel2} for domains in $\mathbb{C}^n$ and these notions were generalized to complex manifolds by S. Kobayashi in \cite{GeometryboundeddomKobayashi1959}. The invariance of the Bergman metric under biholomorphisms makes it a natural choice of metric in complex geometry (see \cite{KrantzBergmankernelbook2013} for details). Since their introduction, questions related to the Bergman kernel and metric have unfurled new directions of research. In this paper we investigate a classification problem of Stein spaces involving the Bergman metric. 
	
	For any bounded strongly pseudoconvex domain there exists a natural complete K\"ahler-Einstein metric, whose existence was proved by S.Y. Cheng and S.T. Yau \cite{ChengYaumetric1980}. The Cheng-Yau metric can be constructed explicitly by solving a Monge-Ampere type equation. This metric is also unique up to scaling by a constant. Yau conjectured that the Cheng-Yau metric of a bounded pseudoconvex domain coincides with its Bergman metric if and only if the domain is homogeneous \cite{Semdiffgeo1982}; recall that a domain in $\mathbb{C}^n$ is called \textit{homogeneous} if its automorphism group acts transitively on it. This is still an open conjecture in complex geometry. 
	
	A  homogeneous strongly pseudoconvex bounded domain with a smooth boundary is biholomorphic to the unit ball $\mathbb{B}^n$ \cite{Wong1977}. Fu-Wong \cite{FuWongChengConjdim2}, Nemirovski-Shafikov \cite{Nemirovski_2006_ConjChengRama} and Huang-Xiao \cite{HuangXiao2021Chengdim3} proved an older conjecture of Cheng \cite{ChengConjectureoriginal} which states that Yau's conjecture holds for strongly psuedoconvex domains with smooth boundary. Huang-Li \cite{huang2020bergmaneinstein} then extended this conjecture to Stein manifolds and proved the following:  \textit{The only Stein manifold with smooth and compact strongly pseudoconvex boundary for which the Bergman metric is K\"ahler-Einstein is the unit ball $\mathbb{B}^n$ up to biholomorphisms}. Furthermore, Huang-Xiao formulated the following conjecture for  Stein spaces with isolated normal singularities:

	\begin{conjecture}[\cite{HuangXiaoUniformization2020}]
		Let $\Omega$ be a normal Stein space with a compact, smooth, strongly pseudoconvex boundary. Then the Bergman metric on the regular part of $\Omega$ is K\"ahler-Einstein if and only if $\Omega$ is biholomorphic to the unit ball in a complex Euclidean space. 
	\end{conjecture}

	In this paper, we prove the following algebraic version of Cheng's Conjecture for normal Stein spaces in dimension two: 
	\begin{theorem}\label{thm:Chengalgdim2}
		Let $\Omega$ be a two dimensional Stein space with isolated normal singularities and a compact, smooth, strongly pseudoconvex boundary $\partial \Omega$. Assume that $\partial \Omega$ is CR equivalent to an algebraic CR manifold in a complex Euclidean space. The Bergman metric on the regular part of the Stein space $\Omega$ is K\"ahler-Einstein if and only if $\Omega$ is biholomorphic to $\mathbb{B}^2$.
	\end{theorem}

	In Corollary 4.1 of \cite{huang2020bergmaneinstein}, Huang and Li showed that it suffices to consider $\Omega$ in Theorem \ref{thm:Chengalgdim2} to be a ball quotient $\mathbb{B}^n/\Gamma$ where $n\ge2$, and $\Gamma\subset Aut(\mathbb{B}^n)$ is finite and fixed point free.  Ebenfelt-Xiao-Xu in \cite{ebenfel2020classificationstein} proved that such ball quotients $\mathbb{B}^n/\Gamma$ do not admit a Bergman metric that is K\"ahler-Einsten provided that $\Gamma$ is nontrivial and also abelian.  In two dimensions we prove the following theorem, which implies Theorem \ref{thm:Chengalgdim2}: 
	\begin{theorem}\label{thm:3}
		For any fixed point free, finite subgroup $\Gamma$ of $Aut(\mathbb{B}^2)$, the regular part of the ball quotient $\mathbb{B}^2/\Gamma$ has a Bergman metric that is K\"ahler-Einstein if and only if $\Gamma =\{id\}$.
	\end{theorem}  
	We emphasize that there is a contrasting result in dimension one. To be precise, the Bergman metric on the regular part of any finite disk quotient $\mathbb{B}^1/\Gamma$ is K\"ahler-Einstein \cite{huang2020bergmaneinstein}.
	
	The remaining sections of the paper are organized as follows: We provide some known results and background material in Section \ref{sec:ballquo}. In Section \ref{sec:KEdim2}, we specialize to dimension 2 and obtain a necessary condition for the Bergman metric on ball quotients to be K\"ahler-Einstein. In Section \ref{sec:grouprep}, we list all the fixed point free, finite groups that we need to consider and find representations of these groups. In the subsequent sections we explicitly show that the necessary condition is not satisfied for the nontrivial groups in the list, thus proving Theorem \ref{thm:3}.  
	
	\subsection*{Acknowledgement}
	The first author would like to thank Prof. Peter Ebenfelt, Prof. Ming Xiao, Prof. Hang Xu, for giving the motivation and all the necessary guidance to develop this project. The same author is also grateful to N. Ramachandran and A. Patil for many enlightening conversations. The second author would like to thank Prof. Dragos Oprea for the interest shown in this work.

	\section{Preliminaries}\label{sec:ballquo}
	In this section we briefly discuss an `effective' version of the K\"ahler-Einstein equation for finite ball quotients. For details, we refer the reader to \cite{ebenfel2020classificationstein} (sections 2.2 and 4.1). 
	
	A K\"ahler-Einstein metric is locally the complex Hessian of a potential function that obeys Einstein's equation, i.e. the Ricci curvature is a constant multiple of the metric tensor. Let $(M,g)$ be a K\"ahler manifold where $g$ is induced locally by the K\"ahler form  $\partial \bar{\partial}\log(\phi)$ such that $\phi > 0$ is smooth. Then the K\"ahler-Einstein equation for $(M,g)$ is: 
	\begin{align}\label{eq:einsteineqforkahlermanifold}
		-\partial\bar{\partial}\log(\Phi)= c\ \partial\bar{\partial}\log(\phi) ,    
	\end{align} 
	where $\Phi=det(g_{i\bar{j}})$ with $g_{i\bar{j}}=\partial_{z_i}\partial_{\overline{z_j}}\log \phi(z,\bar{z})$. It is well known that the constant $c$ has to be negative. 
	
	Recall that  Bergman Kernel form on a manifold (or the regular part of a Stein space) $M$ is an $(n,n)$ form and hence on a local chart the Bergman kernel can be written as:  \begin{align*}
		K_M(z, \bar{w})= k_M(z, \bar{w}) \ dz_1\wedge\cdots \wedge dz_n\wedge d\bar{w}_1\wedge\cdots \wedge d\bar{w}_n.
	\end{align*}
	The Bergman metric is given by the K\"ahler form $\omega=\partial \bar{\partial} \log k_M(z, \bar{z})$.
	
	Let us fix $n \in \mathbb{N}$. The holomorphic automorphism group of the unit ball, $Aut(\mathbb{B}^n)$, is isomorphic to $PSU(n,1)$, the subgroup of the automorphism group of $\mathbb{CP}^n$ with each element in the subgroup fixing $\mathbb{B}^n \subset \mathbb{CP}^n$. $Aut(\mathbb{B}^n)$ is described explicitly in section 2.2 of \cite{RudinFunctiontheoryunitball2008}. 
	
	A subgroup $\Gamma$ of $Aut(\mathbb{B}^n)$ is called fixed point free  if any non-identity element in $\Gamma$ has no fixed point on $\partial \mathbb{B}^n$. By this fixed point free condition, the action of $\Gamma$ on $\partial\mathbb{B}^n$ is properly discontinuous and hence $\partial \mathbb{B}^n/\Gamma$ is a smooth manifold. We know from basic Lie group theory that $U(n)$ is a maximal compact subgroup of $Aut(\mathbb{B}^n)$. This implies that any finite subgroup $\Gamma$ of $Aut(\mathbb{B}^n)$ equals $\psi^{-1} \cdot \tilde{\Gamma}\cdot \psi$ where $\psi \in Aut(\mathbb{B}^n)$ and $\tilde{\Gamma}$ is a subgroup of  $U(n)$. Note that the ball quotients $\mathbb{B}^n/\Gamma$ and $\mathbb{B}^n/\tilde{\Gamma}$ are biholomorphic. Since the Bergman metric is invariant under biholomorphisms, we may assume $\Gamma \subset U(n)$. 
	
	Let us now fix a finite, fixed point free, unitary group $\Gamma$ and let $\Pi: \mathbb{B}^n \to \mathbb{B}^n/\Gamma$ be the quotient map. Note that a function (form) $f_{\Gamma}$ on $\mathbb{B}^n/\Gamma$ is equivalent to a function (form) $f$ on $\mathbb{B}^n$ invariant under $\Gamma$ i.e. $f \circ \gamma=\gamma$ for all $\gamma \in \Gamma$. So we can express the Bergman kernel form (hence the Bergman metric) on $\mathbb{B}^n/\Gamma$ in the coordinates of $\mathbb{B}^n$ as a $\Gamma$ invariant form (or metric). Also, the quotient map $\Pi$ is a holomorphic branched covering map and so is a local biholomorphism between $\mathbb{B}^n \setminus \{0\}$ and  the regular part of $\mathbb{B}^n/\Gamma$. Being K\"ahler-Einstein is a local property of a metric making it invariant under local biholomophisms. So the Bergman metric on $\mathbb{B}^n/\Gamma$ is K\"ahler-Einstein if and only if that metric, expressed in coordinates of $\mathbb{B}^n$ (i.e. the pullback of that metric under $\Pi$) is K\"ahler-Einstein. 
	
	In \cite{ebenfel2020classificationstein}, the Bergman kernel form on $\mathbb{B}^n/\Gamma$ is explicitly expressed in the coordinates of $\mathbb{B}^n$ as: 
	\begin{align*}
		K_{\mathbb{B}^n/\Gamma}(z, \bar{w})=\frac{n!}{\pi^n}\phi(z, \bar{w}) \ dz_1\wedge \cdots \wedge dz_n\wedge d\bar{w}_1\wedge \cdots \wedge d\bar{w}_n,
	\end{align*}
	where $\phi$ is the $\Gamma$ invariant function given by  
	\begin{align*}
		\phi(z, \bar{w}):=\sum_{\gamma \in \Gamma} \frac{\overline{\det{\gamma}}}{(1-<z,\overline{\gamma w}>)^{n+1}}.
	\end{align*}
	where $<u,v>=\sum_{i=1}^n u_i v_i$ for any $u,v \in \mathbb{C}^n$. The Bergman metric on $\mathbb{B}^n/\Gamma$ pulled back to $\mathbb{B}^n$ is given locally by $\partial_{z_i}\partial_{\overline{z_j}}\log \phi(z,\bar{z})$.
	
	From the preceding discussion we know that Bergman metric of our realization of ball quotient $\mathbb{B}^n/\Gamma$ is K\"ahler-Einstein if and only if the function $\phi$ obeys \eqref{eq:einsteineqforkahlermanifold}. In \cite{ebenfel2020classificationstein}, this criterion has been further simplified to: 
	\begin{align}\label{eq: KEBnquover1}
		\Phi(z,\bar{z})=(n+1)^n \phi(z,\bar{z}).    
	\end{align} 
	The well known identity $\Phi(z, \bar{z})= J(\phi(z,\bar{z}))/\phi(z, \bar{z})^{n+1}$ simplifies \eqref{eq: KEBnquover1} into:
	\begin{align*}
		J(\phi(z,\bar{z}))=(n+1)^n \phi(z, \bar{z})^{n+2},
	\end{align*}
	where $J$ is the Monge-Ampere type operator (which differs from the operator $J$ introduced in  \cite{FeffermanMongeAmpere1976,FeffermanErrataMongeAmpere1976} by a sign) given by
	\begin{align*}
		J(\psi):=\det\begin{pmatrix}
			\psi & \psi_{\overline{z_j}} \\ \psi_{z_i} & \psi_{z_i\overline{z_j}}
		\end{pmatrix}.
	\end{align*}
	
	Now, by explicit calculations in \cite{ebenfel2020classificationstein} the authors have shown
	\begin{equation*}
		J(\phi)(z,\bar{z})=(n+1)^n\sum_{\gamma^0,\gamma^1,...,\gamma^n \in \Gamma}\det{A(\overline{\gamma^0},\overline{\gamma^1},...,\overline{\gamma^n})}\prod_{i=0}^{n}\frac{\overline{\det{\gamma^i}}}{(1-<z,\overline{\gamma^iz}>)^{n+2}},
	\end{equation*}
	where $A(\gamma^0,\gamma^1,...,\gamma^n)= (\xi_0(\gamma^0) \ \ \xi_1(\gamma^1)\ \cdots \ \xi_n(\gamma^n))$ is an $(n+1) \times (n+1)$ matrix with column vectors
	\begin{align*}
		\xi_0(\gamma)(z,\bar{z})= \begin{pmatrix} 1-<z,\gamma \bar{z}> \\ (n+1)\gamma \bar{z} \end{pmatrix} \hspace{1cm} \text{and}  \hspace{1cm}  \xi_j(\gamma)(z,\bar{z})= \begin{pmatrix} z^T(\gamma)_j \\ \frac{(\gamma)_j(1-<z,\gamma\bar{z}>)+(n+2)\gamma\bar{z}(z^T(\gamma)_j)}{1-<z,\gamma\bar{z}>} \end{pmatrix},
	\end{align*}
	for $j \in \{1,2,...n\}$. Here, $(\gamma)_j$ is the j-th column vector of a matrix $\gamma \in \Gamma$. Combining all the equations here we get a working version of K\"ahler-Einstein equation for the ball quotient $\mathbb{B}^n/\Gamma$:  
	\begin{align} \label{eq:KEBnquo}
		\sum_{\gamma^0,\gamma^1,...,\gamma^n \in \Gamma}\det{A(\overline{\gamma^0},\overline{\gamma^1},...,\overline{\gamma^n})}\prod_{i=0}^{n}\frac{\overline{\det{\gamma^i}}}{(1-<z,\overline{\gamma^iz}>)^{n+2}}=\left( \sum_{\gamma \in \Gamma} \frac{\overline{\det{\gamma}}}{(1-<z,\overline{\gamma z}>)^{n+1}} \right)^{n+2} 
	\end{align}
	where $<z,\bar{z}>=|z|^2 <1$.

	\section{K\"ahler-Einstein equation in dimension 2}	\label{sec:KEdim2}
	Let $\Gamma$ be a fixed point free subgroup of $U(2)$. In this section, we obtain a condition on $\Gamma$ necessary for Bergman metric on $\mathbb{B}^2/\Gamma$ to be K\"ahler-Einstein. In particular, we define an expression $C(\Gamma)$ in  Proposition \ref{lem:def_C(Gamma)} that must vanish.
	
	\begin{lemma}\label{lemma:fixedpointfreeidmat}
		$\Gamma$ be a subgroup of $U(n), n \geq 2$. If $\Gamma$ is fixed point free then the only $\gamma \in \Gamma$ with $\gamma_{11}=1$ is identity. 
	\end{lemma}
	\begin{proof}
		From the definition of `fixed point free' above, one can prove that a unitary group $\Gamma$ is `fixed point free' if and only if any of the following equivalent conditions hold: (i) any non-identity element in $\Gamma$ has no eigenvalue equal to 1; (ii) $0 \in \mathbb{B}^n$ is the only fixed point for any non-identity element in $\Gamma$.
		
		Since $\Gamma$ is a fixed point free subgroup of $U(n)$, if $\gamma_{11}=1$ for some $\gamma \in \Gamma$ then $\gamma_{1j}=0=\gamma_{j1}$ for $j \in \{2,3,...,n\}$. In particular, 1 becomes an eigenvalue for $\gamma$. In that case, the fixed point free condition on $\Gamma$ ensures that $\gamma=id$.
	\end{proof}
	We specialize the K\"ahler-Einstein equation \eqref{eq:KEBnquo} in dimension $n=2$ and substitute $z=(z_1,0)$ and let $x=z_1\overline{z_1}$. The left hand side of the equation \eqref{eq:KEBnquo} is given by
	\begin{align}\label{eq:KEjphiside}
		f(x) :=\sum_{\gamma^0,\gamma^1,\gamma^2 \in \Gamma}\det{A(\overline{\gamma^0},\overline{\gamma^1},\overline{\gamma^2})}\prod_{i=0}^{2}\frac{\overline{\det{\gamma^i}}}{(1-\overline{\gamma^i_{11}}x)^{4}}
	\end{align}
	and the right hand side of \eqref{eq:KEBnquo} equals $\phi(x)^4$ where
	\begin{equation}\label{eq:defphix}
		\phi(x)=\sum_{\gamma \in {\Gamma}} \frac{\overline{\det{\gamma}}}{(1-<z,\overline{\gamma z}>)^3}=\sum_{\gamma \in {\Gamma}} \frac{\overline{\det{\gamma}}}{(1-\overline{\gamma_{11}}x)^3}.
	\end{equation}
	We note that the determinant of the matrix $A(\overline{\gamma^0},\overline{\gamma^1},\overline{\gamma^2})$ in \eqref{eq:KEjphiside} is
	\begin{equation}\label{eq:detA}
		\det{A(\overline{\gamma^0},\overline{\gamma^1},\overline{\gamma^2})}=\det{\begin{pmatrix}
				1-\overline{\gamma^0_{11}}x & \overline{\gamma^1_{11}} & \overline{\gamma^2_{12}}\\
				3\overline{\gamma^0_{11}}x	& \frac{\overline{\gamma^1_{11}}(1-\overline{\gamma^1_{11}}x)+4(\overline{\gamma^1_{11}})^2x}{1-\overline{\gamma^1_{11}}x} &  \frac{\overline{\gamma^2_{12}}(1-\overline{\gamma^2_{11}}x)+4\overline{\gamma^2_{11}}\overline{\gamma^2_{12}}x}{1-\overline{\gamma^2_{11}}x} \\
				3\overline{\gamma^0_{21}}x & \frac{\overline{\gamma^1_{21}}(1-\overline{\gamma^1_{11}}x)+4\overline{\gamma^1_{11}}\overline{\gamma^1_{21}}x}{1-\overline{\gamma^1_{11}}x} & \frac{\overline{\gamma^2_{22}}(1-\overline{\gamma^2_{11}}x)+4\overline{\gamma^2_{21}}\overline{\gamma^2_{12}}x}{1-\overline{\gamma^2_{11}}x}
		\end{pmatrix}}.
	\end{equation}

	\begin{proposition}\label{lem:def_C(Gamma)}
		Let $\Gamma$ be a fixed point free subgroup $U(2)$. If the Bergman metric on the ball quotient $\mathbb{B}^2/\Gamma$ is K\"ahler-Einstein, then    
		\begin{align*}
			C(\Gamma):=\sum_{\gamma \in {\Gamma\backslash \{id\} }}\frac{\det{\gamma}}{(1-\gamma_{11})^4} \bigg(1-3\tr{\gamma}+\frac{4(\tr{\gamma}-\det{\gamma}-1)}{(1-\gamma_{11})}\bigg)=0.
		\end{align*}
	\end{proposition}
	\begin{proof}
		The K\"ahler-Einstein equation implies that  $f(x)-\phi^4(x)=0$ for $0\le x< 1$. Note that  $f(x)-\phi^4(x)$ is a rational function in $x$, hence it must equal zero as a rational function.  Moreover, we may extend it to a meromorphic function, and consider the Laurent series expansion at $x=1$. We show that $f(x)-\phi^4(x)$ has a pole at 1 of order at most 8 with $C(\Gamma)$ being the coefficient of $1/(x-1)^8$.
		
		First, we find the principal part for Laurent expansion of the function $\phi^4(x)$ around 1. For notational convenience, we work with conjugates of $\phi^4(x) $ and $f(x)$. Using \eqref{eq:defphix} and Lemma \ref{lemma:fixedpointfreeidmat} we get: 
		
		\begin{equation}\label{eq:phi^4}
			\overline{\phi^4(x)}=\left(\frac{1}{(1-x)^3} + \sum_{\gamma \in {\Gamma}, \gamma \neq id} \frac{\det{\gamma}}{(1-\gamma_{11}x)^3} \right)^4.
		\end{equation}
		
		Let us denote 
		\begin{equation*}
			s(x) := \sum_{\gamma \in {\Gamma}, \gamma \neq id} \frac{\det{\gamma}}{(1-\gamma_{11}x)^3}.
		\end{equation*}
		
		Using binomial expansion of \eqref{eq:phi^4} and Maclaurin expansion of $s(x)$ at 1, we obtain the first few terms of the principal part of $\overline{\phi^4(x)}$:
		\begin{align*}
			&\frac{1}{(x-1)^{12}}-\frac{4 s(1)}{(x-1)^{9}} -\frac{4s'(1)}{(x-1)^8} + \cdots \\
			=&\frac{1}{(x-1)^{12}}-\frac{4}{(x-1)^{9}}\sum_{\gamma \neq id} \frac{ \det{\gamma}}{(1-\gamma_{11})^3} -\frac{12}{(x-1)^8}\sum_{\gamma \in {\Gamma}, \gamma \neq id} \frac{\gamma_{11}. \det{\gamma}}{(1-\gamma_{11})^4} + \cdots 
			\label{phi^4asymp}
		\end{align*}
		The calculation for the Laurent expansion requires careful case work. We do the computation in the next Lemma. We obtain the required result by subtracting the principal part for $\overline{\phi^4(x)}$ and $\overline{f(x)}$.
	\end{proof}
	\begin{lemma}
		The principal part for the Laurent expansion of $\overline{f(x)}-\overline{\phi^4(x)}$ at 1 equals
		\begin{align*}
			\frac{C(\Gamma)}{(x-1)^8}+\cdots
		\end{align*}
	\end{lemma}
	\begin{proof}
		Recall the definition
		\begin{align*}
			\overline{f(x)}=\sum_{\gamma^0,\gamma^1,\gamma^2 \in \Gamma}\det{A(\gamma^0,\gamma^1,\gamma^2)}\prod_{i=0}^{2}\frac{\det{\gamma^i}}{(1-\gamma^i_{11}x)^{4}},
		\end{align*}
		We are interested here to only find the coefficients of $1/(x-1)^{\ell}$ where $\ell \geq 8$. Observe from \eqref{eq:detA} that $\det(A(\overline{\gamma^0},\overline{\gamma^1},\overline{\gamma^2}))$ has a pole at 1 of order at most 2. Thus we only need to consider the cases where at least two among the three matrices $\gamma^0,\gamma^1,\gamma^2$ are identity.
		
		\begin{itemize}
			\item $\gamma^0=\gamma^1=\gamma^2=id$: In this case, $\det(A(\gamma^0,\gamma^1,\gamma^2))=1$. So the contribution in $f(x)$ is $1/(x-1)^{12}$.  	 
			
			\item $\gamma^0=\gamma^1=id$, $\gamma^2 \neq id$: In this case, 
			\begin{align*}
				\det(A(\gamma^0,\gamma^1,\gamma^2))=\gamma^2_{22}+\frac{4 \gamma^2_{21} \gamma^2_{12}x}{1-\gamma^2_{11}x}.
			\end{align*}
			So the contribution in the principal part of the Laurent expansion up to 8th order is: 
			\begin{align*}
				&\sum_{\gamma \in \Gamma\setminus \{id\}} \frac{\det{\gamma}}{(1-x)^8(1-\gamma_{11}x)^4}\bigg(\gamma_{22}+\frac{4 \gamma_{21} \gamma_{12}x}{1-\gamma_{11}x} \bigg)\\
				&= \frac{1}{(x-1)^{8}}\sum_{\gamma \in \Gamma\setminus id} \frac{\det{\gamma}}{(1-\gamma_{11})^4} \bigg(\gamma_{22}+\frac{4 \gamma_{21} \gamma_{12}}{1-\gamma_{11}} \bigg)+\cdots 
			\end{align*}

			\item $\gamma^0=\gamma^2=id$, $\gamma^1 \neq id$: In this case, 
			\begin{align*}
				\det(A(\gamma^0,\gamma^1,\gamma^2))=(1-x)\bigg(\gamma^1_{11}+\frac{4 (\gamma^1_{11})^2 x}{1-\gamma^1_{11}}\bigg)-3 x \gamma^1_{11}
			\end{align*}
			So the contribution in the principal part of the Laurent expansion up to 8th order is: 
			\begin{align*}
				&\sum_{\gamma \in \Gamma\setminus \{id\}} \frac{\det{\gamma}}{(1-x)^8(1-\gamma_{11}x)^4 }\bigg((1-x)(\gamma_{11}+\frac{4 (\gamma_{11})^2 x}{1-\gamma_{11}})-3 x \gamma_{11}\bigg)\\
				&= \frac{1}{(x-1)^8}\sum_{\gamma \in \Gamma\setminus id}  \frac{-3 \gamma_{11} \det{\gamma} }{(1-\gamma_{11})^4} +\cdots
			\end{align*}
			
			\item $\gamma^2=\gamma^1=id$, $\gamma^0 \neq id$: In this case,
			\begin{align*}
				\det(A(\gamma^0,\gamma^1,\gamma^2))=\frac{1+(3-4\gamma^0_{11})x}{1-x}
			\end{align*}
			So the contribution in the principal part of the Laurent expansion up to 8th order is: 
			\begin{align*}
				&\sum_{\gamma \in \Gamma\setminus \{id\}} \frac{(1+(3-4\gamma_{11})x)\det{\gamma}}{(1-x)^9 (1-\gamma_{11}x)^4}\\
				&= \frac{1}{(x-1)^9}\sum_{\gamma \in \Gamma\setminus id}  \frac{-4 \det{\gamma}}{(1-\gamma_{11})^3} - \frac{1}{(x-1)^8}\sum_{\gamma \in \Gamma\setminus id} \frac{ \det{\gamma}(3+12 \gamma_{11})}{(1-\gamma_{11})^4} +\cdots
			\end{align*}   	 
		\end{itemize}  
		After combining all these contributions together and using the expansion of $\overline{\phi^4(x)}$, we observe that the coefficient of $(x-1)^{-12}$ and $(x-1)^{-9}$ in $\overline{f(x)}-\overline{\phi^4(x)}$ equals zero and the coefficient of $(x-1)^{-8}$ equals
		\begin{align*}
			\sum_{\gamma \in {\Gamma\backslash \{id\}}}	\frac{\det \gamma}{(1-\gamma)^4}\bigg(\gamma_{22}-3\gamma_{11}-3 +\frac{4\gamma_{12}\gamma_{21}}{1-\gamma_{11}}\bigg)=C(\Gamma).
		\end{align*}
	\end{proof}
	
	\section{Milnor's Theorem and Group Representations}\label{sec:grouprep}
	The following theorem by J. Milnor \cite{MilnoractionS3} gives us all finite, fixed point free groups having a unitary action on $\partial \mathbb{B}^2$. The classification of such finite, fixed point free, unitary groups in higher dimensions can be found in \cite{WolfSpacesofconstantcurvature2011} in the context of finding spaces of constant curvature.
	\begin{theorem*}[\cite{MilnoractionS3}]\label{MilnoractionS3}
		The following is a list of all finite groups which act orthogonally on $ \mathbb{S}^3$ without fixed points: 
		\begin{enumerate}
			\item The group 1;
			\item $ Q_{8n}= <X,Y: X^2=(XY)^2=Y^{2n}>, n \in \mathbb{N}$;
			\item $ P_{48}= <X,Y: X^2=(XY)^3=Y^{4}, X^4=id>$;
			\item $ P_{120}= <X,Y: X^2=(XY)^3=Y^{5}, X^4=id>$;
			\item $ D_{2^m(2n+1)}=<X,Y: X^{2^m}=id, Y^{2n+1}=id, XYX^{-1}=Y^{-1}>$ where $m \geq 2, n \geq 1$;
			\item $P'_{8.3^m}=<X,Y,Z: X^2=(XY)^2=Y^2, ZXZ^{-1}=Y, ZYZ^{-1}=XY, Z^{3^m}=id>$ where $m \in \mathbb{N}$;
			\item Direct product of any of the groups above with a cyclic group of relatively prime order.   
		\end{enumerate}		
		
	\end{theorem*}
	These groups above are sufficient to consider because the real sphere $\mathbb{S}^3$ forms the boundary of the two complex dimensional unit ball $\mathbb{B}^2$ and we are interested in unitary action of a finite, fixed point free group on the unit ball which is a special case of orthogonal action (because $U(2)$ is embeddable in $SO(4)$).

	\begin{remark}
		The representation of any of these groups above, in $U(2)$, has to be faithful in this case, to obey the `fixed point free' condition. Recall from Section \ref{sec:ballquo} that a group $\Gamma\subset Aut(\mathbb{B}^2)$ and its conjugate with an element in $Aut(\mathbb{B}^2)$ produce biholomorphic ball quotients. So it is enough to consider the simplest unitary representations of the groups up to conjugation.
	\end{remark}
	In \cite{ebenfel2020classificationstein}, the authors have already taken care of the cyclic groups. In the following sections we consider the remaining non-abelian groups in the list and their direct products with cyclic groups.\\
	
	\subsection{Representation of the groups in $U(2)$:}\label{repsubsection}
	\begin{proposition} \label{repgamma2}
		Any embedding of the group $Q_{8n}=<X,Y: X^2=(XY)^2=Y^{2n}>$  in $U(2)$ is represented (up to conjugation) by the matrix group generated by
		\begin{align*}
			X=\begin{pmatrix} 0 & 1\\ -1 & 0 \end{pmatrix},\hspace{1cm}  Y= \begin{pmatrix} a & 0\\ 0 & a^{-1} \end{pmatrix}
		\end{align*} 
		where $a$ is a primitive $4n^{\text{th}}$ root of unity. 
	\end{proposition}
	\begin{proof}
		We may conjugate the group to assume that $Y$ is represented by a diagonal matrix. Using the presentation, we observe that $Y$ has order $4n$ and $\det Y\in\{\pm1\}$, hence \[Y=\begin{pmatrix}
			a &0\\ 0 & \pm a^{-1}
		\end{pmatrix}, \]
		where $a$ is primitive $4n^{\text{th}}$ root of unity. Taking $X$ as a general element of $U(2)$, and using the relation $Y^{-1}X=XY$ and $\det X^2=1$, we obtain that $X$ and $Y$ are given by \[X=\begin{pmatrix}
			0 &b\\ - \overline{b} &  0
		\end{pmatrix},\hspace{1cm } Y=\begin{pmatrix}
			a &0\\ 0 &  a^{-1}
		\end{pmatrix}, \]
		where $|b|^2=1$ and $a$ is primitive $4n^{\text{th}}$ root of unity. We may further conjugate the group by the diagonal matrix $\text{Diag}(\sqrt{b}, \sqrt{\overline{b}})$ (where square root is defined using any branch) to obtain the required representation.
	\end{proof}

	\begin{proposition}\label{repgamma5}
		Any embedding of the group $D_{2^{m}(2n+1)}=<X,Y: X^{2^m}=id, Y^{2n+1}=id, XYX^{-1}=Y^{-1}> $, where $m \geq 2, n \geq 1$, in $U(2)$ is represented (up to conjugation) by the matrix group generated by 	\begin{align*}
			X=\begin{pmatrix} 0 & 1\\ b &0 \end{pmatrix},\hspace{1cm} Y= \begin{pmatrix} a & 0\\  0& a^{-1} \end{pmatrix}
		\end{align*} 
		where $b$ is a primitive $2^{m-1}$th root of unity and $a$ is a primitive $(2n+1)^{th}$ root of unity.
		
	\end{proposition}
	\begin{proof}
		We may assume that $Y$ is represented by a diagonal matrix. Moreover, the given relations imply that $\det Y=1$ and has order $(2n+1)$, hence $Y$ is given by the required matrix.
		
		Taking $X$ as a general element in $U(2)$, the relation $XY=Y^{-1}X$ implies that $X$ is an anti-diagonal matrix and is given by \begin{align*}
			X= \begin{pmatrix} 0 & c\\ b\bar{c} &0 \end{pmatrix}.
		\end{align*}  
		where $|c|=|b|=1$.
		We may assume $c=1$ by conjugating the group with diagonal matrix $\text{Diag}(\sqrt{c},\sqrt{\bar{c}})$. Since the order of $X$ is $2^m$, we obtain $b$ is a primitive $2^{m-1}$th root of unity.
	\end{proof}

	\begin{proposition}\label{repgamma34}
		Any fixed point free embedding of the group $P_{48}$ (or $P_{120}$)  in $U(2)$ is represented (up to conjugation) by the matrix group generated by
		\begin{align*}
			X= \begin{pmatrix} c& d\\ -\bar{d} &\bar{c} \end{pmatrix},\hspace{1cm} Y= \begin{pmatrix} a & 0\\  0& a^{-1} \end{pmatrix}
		\end{align*} 
		where $a$ is a primitive $8^{th}$ (or $10^{th}$) root of unity, $c=1/(a-\bar{a})$ and $|d|^2=1+c^2$. 
	\end{proposition}
	\begin{proof}
		We use the group presentation $P_{48}=<X,Y: X^2=(XY)^3=Y^{4}, X^4=id> $.
		We may conjugate the group and assume that $Y$ is represented by a diagonal matrix. Using the presentation of the group, we obtain that $Y$ has order $8$ and $\det X=\det Y\in\{\pm 1\}$. 
		
		Note that the matrix $X$ must not be a diagonal matrix, since $X$ does not commute with $Y$.  Thus the minimal polynomial of $X$ must have degree at least 2, hence it must equal the characteristic polynomial of $X$. Since $X$ satisfies $X^4-id=0$, $X$ has exactly two distinct eigenvalues $\{\lambda_1,\lambda_2\}$ from the set $\{\pm1,\pm i\}$. Since $X$ acts on $\partial\mathbb{B}^2$ fixed point freely and $X$ is not the identity matrix, its eigenvalue cannot be 1. Using $\det X\in\{\pm1\}$, we can conclude that $\{\lambda_1,\lambda_2\}=\{i,-i\}$. In particular, $\det Y=\det X=\lambda_1\lambda_2=1$. 
		
		We may thus assume that 	
		\begin{align*}
			X= \begin{pmatrix} c & d\\ -\bar{d} &\bar{c} \end{pmatrix}\hspace{1cm}\text{and} \hspace{1cm} Y= \begin{pmatrix} a & 0\\  0& a^{-1} \end{pmatrix}.
		\end{align*} 
		where $|c|^2+|d|^2=1$ and $a$ is primitive $8^{th}$ root of unity.
		
		Observe that $Y^4=-id =X^2$ which implies $X=-X^{-1}$. It gives us the identity $c+\bar{c}=0$. Furthermore, we use the relation $(XY)^3=Y^4=-I$ to obtain $XYX=-Y^{-1}X^{-1}Y^{-1}$. Equating entries of matrices in the last identity, we obtain 
		\begin{align*}
			ac^2-|d|^2a^{-1}=-a^{-2}\bar{c} \hspace{1cm}\text{and} \hspace{1cm} acd+a^{-1}\bar{c}d=d .	
		\end{align*}
		Note that $d\ne 0$ since $X$ is not diagonal. Thus the second identity implies $c=1/(a-a^{-1})$.
		
		The result for the group $P_{120}$ can be proved verbatim by letting $a$ be a primitive $10^{\text{th}}$ root of unity.
	\end{proof}

	\begin{proposition}\label{repgamma6}
		Any embedding of the group \[P'_{8\cdot 3^{m}}=<X,Y,Z: X^2=(XY)^2=Y^2, ZX=YZ, ZY=XYZ, Z^{3^m}=id>\]  in $U(2)$ is represented (up to conjugation) by the matrix group generated by
		\begin{align*}
			X=\begin{pmatrix} 0 & 1\\ -1 & 0 \end{pmatrix},\hspace{1cm}  Y= \begin{pmatrix} i & 0\\ 0 & -i \end{pmatrix}\hspace{1cm}
			Z=\frac{\beta}{i-1}
			\begin{pmatrix} 1 & -i\\ -1 & -i \end{pmatrix}
		\end{align*} 
		where $\beta$ is primitive $3^{m}${th} root of unity. 
	\end{proposition}
	\begin{proof}
		Observe that the subgroup generated by $X$ and $Y$ is isomorphic to $Q_8$. Thus letting $Y$ be a diagonal matrix (using conjugation), we obtain the required matrix representation for $X$ and $Y$. 
		
		Take a general matrix representation of $Z$ in $U(2)$ given by 
		\begin{align*}
			Z=\begin{pmatrix} a & b\\ -e^{i\theta}\bar{b} & e^{i\theta}\bar{a} \end{pmatrix}
		\end{align*} 
		where $|a|^2+|b|^2=1$. The relation $ZX=YZ$ implies $b=-ia$. Furthermore, the relation $ZY=XYZ$ gives us $a=ie^{i\theta}\bar{a}$. Substituting back in the above identities, we obtain 
		\begin{align*}
			Z=a\begin{pmatrix} 1 & -i\\ -1 & -i\end{pmatrix}.
		\end{align*} 
		Note that $Z^3=(i-1)^3a^3\cdot id$. Let $\beta=a(i-1)$. Since $Z^{3}$ has order $3^{m-1}$, $\beta$ equals a $3^{m}$th primitive root of unity.
	\end{proof}
	
	\begin{remark}
		The analogous representation of the direct product of the groups $\Gamma$ considered above with $\mathbb{Z}/p\mathbb{Z}$ such that $gcd(p, |\Gamma|)=1$ can be easily obtained. Indeed, a generator of $\mathbb{Z}/p\mathbb{Z}$ must be represented by the matrix $u\cdot id$, where $u$ is a primitive $p^{th}$ root of unity. This can be explicitly observed by using the property that the matrix representing generator of $\mathbb{Z}/p\mathbb{Z}$ commutes with the generators of $\Gamma$ and has order $p$. 
	\end{remark}

	\section{The groups $Q_{8n}$ and $D_{2^m(2n+1)}$}
	First we state an important lemma:
	\begin{lemma}\label{lem:SUnonKE}
		If $\Gamma$ is a nontrivial, fixed point free, finite subgroup of  $SU(2)$ then $\mathbb{B}^2/\Gamma$ does not have K\"ahler-Einstein Bergman metric. In particular, it is true for $\Gamma=Q_{8n}$ or $D_{2^m(2n+1)}$.   
	\end{lemma} 
	\begin{proof}
		If $\Gamma$ is a subgroup of $SU(2)$ then Bergman kernel of the ball quotient, $\mathbb{B}^2/\Gamma$ evaluated at the origin, $\phi(0,0)= \text{order of} \ \Gamma \neq 0$ . So by Corollary 5.4 of \cite{huang2020bergmaneinstein}, the Bergman metric of $\mathbb{B}^2/\Gamma$ is not K\"ahler-Einstein. This lemma also generalizes to higher dimensional ball quotients. 
	\end{proof}
	
	\begin{remark}
		The above lemma does not apply to all groups considered in the last section because many of these groups cannot be realized as subgroups of $SU(2)$. If $\Gamma$ is not a subgroup of $SU(2)$ then $\phi(0,0)=0$ and the method shown above, fails. For example: the representations of direct products of nontrivial cyclic groups with $\Gamma=Q_{8n}$ or $D_{2^m(2n+1)}$, obtained in Section \ref{repsubsection}, are not embeddable in $SU(2)$.  
	\end{remark}

	\begin{proposition}\label{prop:gamma2-sum}
		Fix positive integers $n$ and $p \geq 2$  such that $gcd(p,8n)=1$. Let $\Gamma\cong\mathbb{Z}/p\mathbb{Z}\times Q_{8n}$ be the matrix group generated by $u\cdot id$, $X$ and $Y$, where $u$ is a primitive $p^{th}$ root of unity, and $X$ and $Y$ are described in Proposition \ref{repgamma2}. Then 
		\begin{align}\label{eq:psi_def}
			C(\Gamma)=\sum_{\gamma \in {\Gamma\backslash \{id\} }}\psi(\gamma)\ne 0,
		\end{align}
		where 
		\[\psi(\gamma):=\frac{\det{\gamma}}{(1-\gamma_{11})^4} \bigg(1-3\tr{\gamma}+\frac{4(\tr{\gamma}-\det{\gamma}-1)}{(1-\gamma_{11})}\bigg). \]
	\end{proposition}
	\begin{proof}		
		Using the presentation of the group $Q_{8n}$ given in Proposition \ref{repgamma2}, we observe that \[ \Gamma=\{u^{j}Y^kX^\ell:0\le j<p, 0\le k<2n, 0\le \ell<4 \}. \]
		The matrix $u^jY^kX^\ell$ is diagonal for $\ell\in \{0,2\}$ and anti-diagonal when $\ell \in\{1,3 \}$. 
		
		Note that for $\gamma = u^jY^kX^\ell$ where $\ell\in \{1,3\}$, we have $\tr \gamma=\gamma_{11}=0$ and $\det \gamma =u^{2j}$. Since $p$ is odd and $u$ is primitive $p^{th}$ root of unity, for fixed $\ell\in \{1,3\}$:
		\begin{align*}
			\sum_{j,k}^{}\psi(u^jY^kX^{\ell})=\sum_{j,k}^{} u^{2j}(1+4(-u^{2j}-1))=-2n\sum_{j=0}^{p-1}(3u^{2j}+4u^{4j})=0.
		\end{align*}

		We also observe that $X^2=-id=Y^{2n}$, thus $Y^{k}X^2=Y^{k+2n}$. For $0\le j\le p-1$ and $0\le k\le 4n-1$, we have
		\begin{align*}
			\psi(u^jY^k)&=\frac{u^{2j}}{(1-u^ja^k)^4}\bigg(1-3u^j(a^k+a^{-k})+\frac{4(u^{j}(a^k+a^{-k})-u^{2j}-1)}{1-u^{j}a^k} \bigg)\\
			&= \frac{u^{2j}}{(1-u^ja^k)^4}\big(1-3u^j(a^k+a^{-k})+4(u^ja^{-k}-1) \big)\\
			&= \frac{u^{2j}}{(1-u^ja^k)^4}\big(-3-3u^ja^k+u^ja^{-k}\big).
		\end{align*}
		Note that $u^ja^k$ runs over all $N:=4np$ roots of unity. Let $0\le r\le N-1$ be the unique remainder modulo $N$ determined by the Chinese remainder theorem satisfying \[r\equiv\begin{cases}
			& 0 \mod 4n\\
			&2\mod p
		\end{cases}. \] 
		Let $\xi=u^ja^k$ be an $N^{th}$ root of unity, then the choice of $r$ implies $\xi^r=u^{2j}$ and \[\psi(u^jY^k)=\frac{\xi^{r}}{(1-\xi)^4}(-3-3\xi+\xi^{r-1}). \]
		Thus we obtain
		\begin{align*}
			C(\Gamma)=\sum_{\xi\ne 1}\frac{-3\xi^{r}-3\xi^{r+1}+\xi^{2r-1} }{(1-\xi)^4},
		\end{align*}
		where $\xi$ runs over $N^{th}$ roots of unity. Using Lemma \ref{lem:sum_roots_of_unity} given below, we obtain that $720\cdot C(\Gamma)$ is an integer congruent to $-3f(r)-3f(r+1)+f(2r-1)$ modulo $N$ where $f(r): =- 30r^4 + 240r^3 - 660r^2 + 720r - 251$. Observe that $N$ is even, and  $f(r)$, $f(r-1)$ and $f(2r-1)$ are odd. Therefore, $720\cdot C(\Gamma)$ is odd, implying $C(\Gamma)\ne0$.
		
	\end{proof}

	\begin{lemma}\label{lem:sum_roots_of_unity}
		For any positive integer $r$, the sum 
		\begin{align*}
			720\cdot\sum_{\xi\ne 1 }^{}\frac{\xi^r}{(1-\xi)^4},
		\end{align*}
		is an integer congruent to $f(r)=- 30r^4 + 240r^3 - 660r^2 + 720r - 251 \mod N$, where $\xi$ runs over $N^{th}$ roots of unity.
	\end{lemma}
	\begin{proof}
		Let $r\in\{1,\dots , N\}$. Consider the generating function
		\begin{align*}
			\sum_{n\ge 1}^{}z^{n}\sum_{\xi\ne1}\frac{\xi^{r}}{(1-\xi)^n}&=z\sum_{\xi\ne 1}\frac{\xi^r}{1-z-\xi}\\
			&= 1+z\sum_{\xi}\frac{\xi^r}{1-z-\xi}\\
			&= 1+\frac{Nz(1-z)^{r-1}}{(1-z)^N-1}.
		\end{align*}
		A careful calculation (or using computer) shows that the coefficient of $z^4$ in the above power series equals
		\begin{align}\label{eq:coeff_4th}
			F(N,r)=	\frac{N^{4}}{720}  - \left( r^{2} - 4 r +
			\frac{11}{3}\right)\frac{ N^{2}}{24} + \binom{r-1}{3} \frac{N}{2} -\frac{r^{4}}{24}  + \frac{r^{3}}{3}  - \frac{11r^{2}}{12}  + r -
			\frac{251}{720}.
		\end{align}
		The lemma follows by multiplying the above equation by $720$. For $r\notin \{1,\dots,N\}$, we may pick an integer $q$ such that $r+qN\in\{1,\dots,N\}$. We obtain the proof by substituting $r+qN$ in place of $r$ in \eqref{eq:coeff_4th} and taking modulo $N$.
	\end{proof}

	\begin{proposition}
		Fix positive integers $n$, $m\ge 2$ and $p \geq 2$ such that $gcd(p,2^m(2n+1))=1$. Let $\Gamma\cong\mathbb{Z}/p\mathbb{Z}\times D_{2^m(2n+1)}$ be matrix group generated by $u\cdot id$, $X$ and $Y$, where $u$ is a primitive $p^{th}$ root of unity and $X$ and $Y$ are described in Proposition \ref{repgamma5}. Then 
		\begin{align*}
			C(\Gamma)=\sum_{\gamma \in {\Gamma\backslash \{id\} }}\psi(\gamma)\ne 0,
		\end{align*}
		where $\psi$ is defined in \eqref{eq:psi_def}.
	\end{proposition}
	\begin{proof}
		Using the presentation of $D_{2^m(2n+1)}$ given in Proposition \ref{repgamma5}, we obtain
		\[
		\Gamma=\{u^{j}Y^kX^\ell:0\le j<p, 0\le k<2n+1, 0\le \ell<2^m \}. \] 
		When $\ell$ is odd, $\gamma=u^{j}Y^kX^\ell$ is anti-diagonal, thus $\tr \gamma=\gamma_{11}=0$ and $\det \gamma = (-b)^\ell u^j$. Since $p$ is odd, summing \[\psi(u^jY^kX^{\ell})= (-b)^\ell u^{2j}(1-4(1+(-b)^\ell u^{2j}))\]
		over $0\le j\le p-1$ equals zero. Thus only the terms with even exponents of $X$ contribute to $C(\Gamma)$.
		
		Observe that $X^2=b\cdot id$ where $b$ is $2^{m-1}$th root of unity. If we let $\ell=2s$ then $u^jY^kX^{\ell}=u^jb^sY^k$. For a fixed $k$, we may replace the sum of \[\sum_{\substack{0\le j< p\\ 0\le s< 2^{m-1}}}^{}\psi(u^jb^sY^k) = \sum_{0\le j<2^{m-1}p}^{}\psi(\tilde{u}^jY^k) \]
		where $\tilde{u}$ is a primitive $2^{m-1}p$ th root of unity. For $0\le j<2^{m-1} p$ and $0\le k<2n+1$, we have
		\begin{align*}
			\psi(\tilde{u}^jY^k)&= \frac{\tilde{u}^{2j}}{(1-\tilde{u}^ja^k)^4}\big(-3-3\tilde{u}^ja^k+\tilde{u}^ja^{-k}\big).
		\end{align*}
		Note that $\tilde{u}^ja^k$ runs over all $N:=2^{m-1}p(2n+1)$ roots of unity. Let $0\le r\le N-1$ be the unique remainder modulo $N$ determined by the Chinese remainder theorem satisfying \[r\equiv\begin{cases}
			& 0 \mod 2n+1\\
			&2\mod 2^{m-1}p
		\end{cases}. \] 
		Let $\xi=u^ja^k$ be an $N^{th}$ root of unity, then the choice of $r$ implies $\xi^r=u^{2j}$ and \[\psi(u^jY^k)=\frac{\xi^{r}}{(1-\xi)^4}(-3-3\xi+\xi^{r-1}). \]
		Thus we may obtain
		\begin{align*}
			C(\Gamma)=\sum_{\xi\ne 1}\frac{-3\xi^{r}-3\xi^{r+1}+\xi^{2r-1} }{(1-\xi)^4}
		\end{align*}
		where $\xi$ runs over $N^{th}$ roots of unity. Thus $C(\Gamma)\ne 0$ using the exact same arguments in the proof of Proposition \ref{prop:gamma2-sum}.
	\end{proof}
	
	\section{The groups $P_{48}$ and $P_{120}$}
	\begin{proposition}\label{prop_p48sum}
		Fix positive integers $p$ such that $gcd(p,48)=1$. Let $\Gamma\cong\mathbb{Z}/p\mathbb{Z}\times P_{48}$ be the matrix group generated by $u\cdot id$, $X$ and $Y$, where $u$ is a primitive $p^{th}$ root of unity and $X$ and $Y$ are described in Proposition \ref{repgamma34}. Then 
		\begin{align}
			C(\Gamma)=\sum_{\gamma \in {\Gamma\backslash \{id\} }}\psi(\gamma)\ne 0,
		\end{align}
		where 
		$\psi(\gamma)$ is defined in \eqref{eq:psi_def}
	\end{proposition}
	\begin{proof}
		A careful analysis of the representation of the groups given in Proposition \ref{repgamma34} shows that the elements $\gamma \in \Gamma$ for which $|\gamma_{11}|=1$ are precisely the $8p$ elements of the form $\{u^jY^k\}$ where $0\le k<8$ and $0\le j\le p-1$. We split $C(\Gamma)$ as: 
		\begin{align*}
			C(\Gamma)=C_1(p)+ C_2(p),
		\end{align*}
		where 
		\begin{align}\label{c1c2def}
			C_1(p)=\sum_{|\gamma_{11}| \ne 1}\psi(\gamma) \hspace{1cm}\text{and} \hspace{1cm}
			C_2(p)=\sum_{\substack{\gamma \in \Gamma\backslash \{id\} \\ |\gamma_{11}|= 1} }\psi(\gamma).
		\end{align}
		Using the two Lemmas below, we observe that for $p> 26$ we have
		\begin{align*}
			|C(\Gamma)|&> ||C_2(p)|-|C_1(p)||\\
			&>  \bigg|-\frac{28}{9} p^{4}+ i^{p+1} 12 p^{3} + \frac{5}{9} p^{2} - \frac{1}{144}\bigg| -|4.89  \times 10^4 \cdot p | 
			\\
			&> 0.
		\end{align*} 
		For $1\le p\le 26$, we explicitly calculate $C(\Gamma)$ using computer. Please refer to the Appendix \ref{appendix}. 
		
	\end{proof}

	\begin{lemma}\label{lem:C_1(p)}
		For $p\ge 2$ and $gcd(p,48)=1$, we have 
		\begin{align*}
			|C_1(p)| \leq 4.89  \times 10^4 \cdot p .
		\end{align*}
	\end{lemma}
	\begin{proof}	
		By explicit calculation, we note that $(Y^2X)^2$ is an anti-diagonal matrix. We first consider the sum over the $8p$ elements of the form $\gamma=u^j Y^k(Y^2X)^2$ where $j \in \{0,1,\dots,p-1\}$ and $k \in \{0,1,2,\dots,7\}$. We observe that for these elements $\gamma_{11}=\tr \gamma=0$ and hence $\psi(\gamma)=u^{2j}(1-4(-u^{2j}-1))$. Thus
		\begin{align*}
			\sum_{k=0}^{7}\sum_{j=0}^{p-1} \psi(u^j Y^k(Y^2X)^2)=	\sum_{k=0}^{7}\sum_{j=0}^{p-1} u^{2j}(1-4(-u^{2j}-1)) =0.
		\end{align*}
		Here we use that $gcd(p,48)=1$ and $u$ is a primitive $p^{\text{th}}$ root of unity.
		
		Recall the constant $c$ from Proposition \ref{repgamma34}. We use Sage to observe that $|\gamma_{11}|\le |c|$ for $\gamma\in \Gamma$ such that $|\gamma_{11}|\ne 1$. We now find an upper bound for the remaining $32p$ terms appearing in $C_1(p)$ for which $\gamma_{11}\ne 0$. Note that for these elements $\gamma$, we have $\det\gamma=u^{2j}$, $\gamma_{11}=u^j\alpha$ and $\tr \gamma=u^j(\alpha+\bar{\alpha})$ where $|\alpha|\le |c|$. We obtain the following bound using triangle inequality:
		\begin{align*}
			|\psi(\gamma)|&=\bigg|\frac{u^{2j}}{(1-u^j\alpha)^4}\bigg(1-3u^j(\alpha+\overline{\alpha})+4\frac{u^j(\alpha+\overline{\alpha})-u^{2j}-1}{1-u^j\alpha}\bigg)\bigg|\\
			&=\bigg|\frac{u^{2j}}{(1-u^j\alpha)^4}\bigg(-3-3u^j(\alpha+\overline{\alpha})+4\frac{u^j\overline{\alpha}-u^{2j}}{1-u^j\alpha}\bigg)\bigg|\\
			&\le \frac{7+6|c|}{(1-|c|)^4}.
		\end{align*}
		In the last step, we note that the term $(u^j\overline{\alpha}-u^{2j})/(1-u^j\alpha)$ has absolute value 1. Using $|c|=1/\sqrt{2}$, we obtain: 
		\begin{align*}
			|C_1(p)| \leq 32p \frac{7+6|c|}{(1-|c|)^4}\le  4.89 \times 10^4 p.
		\end{align*}
	\end{proof}
	\begin{lemma}
		We have 
		\begin{align*}
			C_2(p)&= -\frac{28}{9} p^{4}+ i^{p+1} 12 p^{3} + \frac{5}{9} p^{2} - \frac{1}{144}.
		\end{align*}
	\end{lemma}
	
	\begin{proof}
		
		We observe that  \[\psi(u^jY^k)= \frac{u^{2j}}{(1-u^ja^k)^4}\big(-3-3u^ja^k+u^ja^{-k}\big).  \]
		
		Note that $u^ja^k$ runs over all $N:=8p$ roots of unity. Let $0\le r\le N-1$ be the unique remainder modulo $N$ determined by Chinese remainder theorem satisfying $r\equiv 0 \mod 8$ and $r\equiv 2\mod p.$ The solutions to this  congruence equation are precisely \[r=\begin{cases}
			6p+2 &  \text{when }p \equiv 1\mod 4 \\
			2p+2 & \text{when }p \equiv 3\mod 4
		\end{cases}. \]
		Let $\xi=u^ja^k$, then the choice of $r$ implies $\xi^r=u^{2j}$ and \[\psi(u^jY^k)=\frac{\xi^{r}}{(1-\xi)^4}(-3-3\xi+\xi^{r-1}). \]
		Thus we may obtain
		\begin{align*}
			C_2(p)=\sum_{\xi\ne 1}\frac{-3\xi^{r}-3\xi^{r+1}+\xi^{2r-1} }{(1-\xi)^4},
		\end{align*}
		where $\xi$ runs over $N^{th}$ roots of unity.
		
		In Lemma \ref{lem:sum_roots_of_unity}, we obtained an explicit polynomial expression in \eqref{eq:coeff_4th} for the sum \[F(N,r)=\sum_{\xi\ne 1}\frac{\xi^r}{(1-\xi)^4},\]
		where $1\le r\le N$. 
		
		When $p\equiv 1\mod 4$, we substitute $N=8p$ and $r\equiv 6p+2\mod p$ to obtain
		\begin{align*}
			C_2(p)&= - 3F(8p,6p+2)-3 F(8p,6p+3)+ F(8p,4p+3)\\
			&= -\frac{28}{9} p^{4} - 12 p^{3} + \frac{5}{9} p^{2} - \frac{1}{144}.
		\end{align*}
		In the last term, we used $2r-1\equiv 4p+3 \mod N$.
		
		When $p\equiv 3\mod 4$, we substitute $r\equiv 2p+2\mod p$ to obtain 
		\begin{align*}
			C_2(p)&= -\frac{28}{9} p^{4} + 12 p^{3} + \frac{5}{9} p^{2} - \frac{1}{144}.
		\end{align*}
		Notice that this expression has opposite sign for the term $12p^3$ compared to previous case.
	\end{proof}
	
	\begin{proposition}\label{prop_p120sum}
		Fix positive integers $p$ such that $gcd(p,120)=1$. Let $\Gamma\cong\mathbb{Z}/p\mathbb{Z}\times P_{120}$ be the matrix group generated by $u\cdot id$, $X$ and $Y$, where $u$ is a primitive $p^{th}$ root of unity and $X$ and $Y$ are described in Proposition \ref{repgamma34}. Then 
		\begin{align}
			C(\Gamma)=\sum_{\gamma \in {\Gamma\backslash \{id\} }}\psi(\gamma)\ne 0,
		\end{align}
		where 
		$\psi(\gamma)$ is defined in \eqref{eq:psi_def}.
	\end{proposition}
	\begin{proof}
		A careful analysis of the representation of the groups given in Proposition \ref{repgamma34} shows that the elements $\gamma \in \Gamma$ for which $|\gamma_{11}|=1$ are precisely the $10p$ elements of the form $\{u^jY^k\}$ where $0\le k<10$ and $0\le j\le p-1$. We split $C(\Gamma)$ as: 
		\begin{align*}
			C(\Gamma)=C_1(p)+ C_2(p),
		\end{align*}
		where 
		\begin{align*}\label{c1c2def}
			C_1(p)=\sum_{|\gamma_{11}| \ne 1}\psi(\gamma) \hspace{1cm}\text{and} \hspace{1cm}
			C_2(p)=\sum_{\substack{\gamma \in \Gamma\backslash \{id\} \\ |\gamma_{11}|= 1} }\psi(\gamma).
		\end{align*}
		Using the two Lemmas below, we observe that for $p> 46$ we have
		\begin{align*}
			|C(\Gamma)|> ||C_2(p)|-|C_1(p)||> 0.
		\end{align*} 
		For $1\le p\le 46$, we explicitly calculate $C(\Gamma)$ using computer, please refer to the Appendix \ref{appendix}. 
		
	\end{proof}

	\begin{lemma}
		For $p\ge 2$ and $gcd(p,120)=1$, we have 
		\begin{align*}
			|C_1(p)| \leq 2.68  \times 10^6 \cdot p.
		\end{align*}
	\end{lemma}
	\begin{proof}
		We use Sage to observe that $|\gamma_{11}|\le |c|$ for $\gamma\in \Gamma$ such that $|\gamma_{11}|\ne 1$. The constant $c$ is described in Proposition \ref{repgamma34}. We now find an upper bound for the $110p$ terms appearing in $C_1(p)$ for which $|\gamma_{11}|\ne 1$. The bound on $\psi(\gamma)$ for such $\gamma$ is obtained in exactly the same way as in the proof of Lemma \ref{lem:C_1(p)}. In particular, we have
		\begin{align*}
			|C_1(p)| \leq 110p \frac{7+6|c|}{(1-|c|)^4}\le 2.68 \times 10^6 p.
		\end{align*}
		Here we used that $|c|=(2\sin(\pi/5))^{-1}$.
	\end{proof}
	
	\begin{lemma}
		We have
		\begin{align*}
			C_2(p)=\begin{cases}
				-\frac{265}{9} \, p^{4} - 20 \, p^{3} + \frac{25}{18} \, p^{2} - \frac{1}{144} & p\equiv 1 \mod 5\\
				\frac{575}{9} \, p^{4} + 20 \, p^{3} - \frac{35}{18} \, p^{2} - \frac{1}{144} & p\equiv 2 \mod 5\\
				\frac{455}{9} \, p^{4} - 8 \, p^{3} - \frac{11}{18} \, p^{2} - \frac{1}{144}& p\equiv 3 \mod 5 \\
				-\frac{265}{9} \, p^{4} - 20 \, p^{3} + \frac{25}{18} \, p^{2} - \frac{1}{144}&p\equiv 4 \mod 5
			\end{cases}.
		\end{align*}
	\end{lemma}
	\begin{proof}
		
		We observe that  \[\psi(u^jY^k)= \frac{u^{2j}}{(1-u^ja^k)^4}\big(-3-3u^ja^k+u^ja^{-k}\big).  \]
		
		Note that $u^ja^k$ runs over all $N:=10p$ roots of unity. Let $0\le r\le N-1$ be the unique remainder modulo $N$ determined by Chinese remainder theorem satisfying $r\equiv 0 \mod 10$ and $r\equiv 2\mod p.$ The explicit value of $r$ is given below: \[r=\begin{cases}
			8p+2 &  \text{when }p \equiv 1\mod 5 \\
			4p+2 & \text{when }p \equiv 2\mod 5\\
			6p+2 &  \text{when }p \equiv 3\mod 5 \\
			2p+2 & \text{when }p \equiv 4\mod 5
		\end{cases}. \]
		Let $\xi=u^ja^k$, then the choice of $r$ implies $\xi^r=u^{2j}$ and \[\psi(u^jY^k)=\frac{\xi^{r}}{(1-\xi)^4}(-3-3\xi+\xi^{r-1}). \]
		Thus we may obtain
		\begin{align*}
			C_2(p)=\sum_{\xi\ne 1}\frac{-3\xi^{r}-3\xi^{r+1}+\xi^{2r-1} }{(1-\xi)^4},
		\end{align*}
		where $\xi$ runs over $N^{th}$ roots of unity.
		
		In Lemma \ref{lem:sum_roots_of_unity}, we obtained an explicit polynomial expression in \eqref{eq:coeff_4th} for the sum \[F(N,r)=\sum_{\xi\ne 1}\frac{\xi^r}{(1-\xi)^4},\]
		where $1\le r\le N$. 
		
		When $p\equiv 1\mod 5$, we substitute $N=10p$ and $r\equiv 8p+2\mod p$ to obtain
		\begin{align*}
			C_2(p)&= - 3F(10p,8p+2)-3 F(10p,8p+3)+ F(10p,6p+3)\\
			&= -\frac{265}{9} \, p^{4} - 20 \, p^{3} + \frac{25}{18} \, p^{2} - \frac{1}{144}.
		\end{align*}
		In the last term, we used $2r-1\equiv 6p+3 \mod N$.
		
		The remaining cases are dealt exactly in the same way.
	\end{proof}

	\section{The groups $P'_{8.3^m}$}
	\begin{proposition}\label{propp'_83powksum}
		Fix positive integer $p$ such that $gcd(p,8\cdot 3^m )=1$. Let $\Gamma\cong\mathbb{Z}/p\mathbb{Z}\times P'_{8.3^m}$ be the matrix group generated by $u\cdot id$, $X$, $Y$ and $Z$, where $u$ is a primitive $p^{th}$ root of unity and $X, \ Y, \ Z$ are described in Proposition \ref{repgamma6}. Then 
		\begin{align*}
			C(\Gamma)=\sum_{\gamma \in {\Gamma\backslash \{id\} }}\psi(\gamma)\ne 0,
		\end{align*}
		where 
		$\psi(\gamma)$ is defined in \eqref{eq:psi_def}
	\end{proposition}
	
	\begin{proof}
		First a careful analysis from the representation of the groups given in Section \ref{repsubsection}, shows that elements of $\Gamma$ can be uniquely written as $u^k\cdot Z^{i_1}\cdot X^{i_2}\cdot Y^{i_3}$ where $k \in \{1,2, \dots ,p\}, i_1 \in \{1,2, \dots ,3^m\}, i_2 \in \{1,2\}$ and $i_3 \in \{1,2,3,4\}$. Then looking at the representation of the group $P'_{8.3^m}$, we note that for $\gamma \in \Gamma$ such that  $|\gamma_{11}| = 1$, we must have $i_1 \in 3 \mathbb{N}$ and $i_2=2$. Then just like last proposition, we can split $C(\Gamma)$ as: 
		\begin{align*}
			C(\Gamma)=C_1(p,m)+ C_2(p,m),
		\end{align*}
		where 
		\begin{align*}\label{c1c2def}
			C_1(p,m)=\sum_{|\gamma_{11}| \ne 1}\psi(\gamma) \hspace{1cm}\text{and} \hspace{1cm}
			C_2(p,m)=\sum_{\substack{\gamma \in \Gamma\backslash \{id\} \\ |\gamma_{11}|= 1} }\psi(\gamma)
		\end{align*}
		Using the two Lemmas below, we observe that for $\tilde{p}=p \cdot 3^{m-1}\geq 24$, we have
		\begin{align*}
			|C(\Gamma)|> ||C_2(p,m)|-|C_1(p,m)||> 0.
		\end{align*} 
		For $1\le \tilde{p} \le 23$, we explicitly calculate $C(\Gamma)$ using computer. Please refer to the Appendix \ref{appendix}. 
		
	\end{proof}

	\begin{lemma}
		For $m, p \in \mathbb{N}$ and $gcd(p,8 \cdot 3^m)=1$, we have 
		\begin{align*}
			|C_1(p,m)| \leq 80(135 \sqrt{2}+ 191)  \cdot 3^{m-1} \cdot p 
		\end{align*}
	\end{lemma}
	\begin{proof}
		Recall that any element $\gamma\in \Gamma$ can be represented as $\gamma=u^k Z^{i_1} X^{i_2} Y^{i_3}$. Using the explicit description of matrices $X$, $Y$ and $Z$, we observe that if $|\gamma_{11}| \neq 1$, then $|\gamma_{11}| \leq \frac{1}{\sqrt{2}}$. For $\gamma\in \Gamma$, we have $\det\gamma=u^{2j}$, $\gamma_{11}=u^j\alpha$ and $\tr \gamma=u^j(\alpha+\bar{\alpha})$ where $|\alpha|\le 1/\sqrt{2}$. Using the bound in the proof of Lemma \ref{lem:C_1(p)}, we have
		\begin{align*}
			|\psi(\gamma)|\le \frac{7+6|\alpha|}{(1-|\alpha|)^4}\le \frac{7+3\sqrt{2}}{(1-1/\sqrt{2})^4}.
		\end{align*}
		Since there are $8p3^m-4p3^{m-1}=20p3^{m-1}$ elements $\gamma$ such that $|\gamma_{11}|\ne 1$, we obtain
		\begin{align*}
			|C_1(p,m)| \leq 20 \cdot 3^{m-1}p\frac{7+3\sqrt{2}}{(1-1/\sqrt{2})^4}= 80(135 \sqrt{2}+ 191)  \cdot 3^{m-1} \cdot p.
		\end{align*}
	\end{proof}
	
	\begin{lemma}
		We have
		\[C_2(p,m)=\sum_{|\gamma_{11}|=1} \psi(\gamma)=\frac{20}{9} \tilde{p}^{4} - \frac{7}{9} \tilde{p}^{2} - \frac{1}{144}, \]
		where $\tilde{p}=p\cdot 3^{m-1}$.
	\end{lemma}
	\begin{proof}
		Note that $|\gamma_{11}|=1$ if and only if $\gamma=u^jZ^{3k}X^2Y^{\ell}=-u^jZ^{3k}Y^{\ell}=u^jZ^{3k}Y^{\ell+2}$. We observe that $Z^3=\beta^3\cdot id$, where $\beta$ is a primitive $3^{m}$th root of unity. Thus

		\[\psi(u^jZ^{3k}Y^\ell)=\psi(u^j\beta ^{3k}Y^\ell).   \]
		
		Since $u^j\beta^{3k}$ runs over $\tilde{p}=p3^{m-1}$ roots of unity, we may rewrite 
		\[\psi(w^jY^k)\frac{w^{2j}}{(1-w^ja^k)^4}\big(-3-3w^ja^k+w^ja^{-k}\big). \]
		
		Note that $w^ja^k$ runs over all $N:=4\tilde{p}$ roots of unity. Let $0\le r\le N-1$ be the unique remainder modulo $N$ determined by Chinese remainder theorem satisfying $r\equiv 0 \mod 4$ and $r\equiv 2\mod \tilde{p}.$ The solutions to this  congruence equation is precisely $r=2\tilde{p}+2 $.
		
		Let $\xi=w^ja^k$, then the choice of $r$ implies $\xi^r=w^{2j}$ and \[\psi(w^jY^k)=\frac{\xi^{r}}{(1-\xi)^4}(-3-3\xi+\xi^{r-1}). \]
		Thus we may obtain
		\begin{align*}
			C_2(p,m)=\sum_{\xi\ne 1}\frac{-3\xi^{r}-3\xi^{r+1}+\xi^{2r-1} }{(1-\xi)^4}
		\end{align*}
		where $\xi$ runs over $N^{th}$ roots of unity.
		
		Using the formula \eqref{eq:coeff_4th} for $F(N,r):=\sum_{\xi\ne 1}\frac{\xi^r}{(1-\xi)^4}$ where $1\le r\le N$, we obtain
		\begin{align*}
			C_2(p,m)&= - 3F(4\tilde{p},2\tilde{p}+2)-3 F(4\tilde{p},2\tilde{p}+3)+ F(4\tilde{p},3)\\
			&= \frac{20}{9} \tilde{p}^{4} - \frac{7}{9} \tilde{p}^{2} - \frac{1}{144}.
		\end{align*}
		
	\end{proof}
	
	\appendix
	\section{}\label{appendix}
	Here we give tables for the explicit coefficient $C(\Gamma)$ for the groups $\Gamma$ mentioned previously in Propositions \ref{prop_p48sum}, \ref{prop_p120sum} and  \ref{propp'_83powksum}. The values of $C(\Gamma)$ below are obtained numerically using MATLAB and Sage. 
	\subsection{$\Gamma= P_{48} \times \mathbb{Z}/p$:}
	Here we give the table of $C(\Gamma)$ for $1 \leq p \leq 26$, with $gcd(p,48)=1$. 
	\begin{center}
		\begin{tabular}{|M{5cm}|M{8cm}|}
			\hline
			Values of $p$  & Values of $C(\Gamma)$\\ 
			\hline
			1          & -5968.348737366    \\ \hline
			5              & -17682.290187324    \\ \hline
			7        &  -5491.7044210122  \\ \hline
			11              & -29599.5039863668  \\ \hline
			13              &  -116823.112487798 \\ \hline
			17            &  -318950.767090781 \\ \hline
			19            &  -322934.584891557  \\ \hline
			13           &   -724318.562730527   \\ \hline
			25          &  -1402436.345364025  \\ \hline	
		\end{tabular}
	\end{center}
	
	\subsection{$\Gamma= P_{120} \times \mathbb{Z}/p$:}
	Here we give the table of $C(\Gamma)$ for $1 \leq p \leq 46$, with $gcd(p,120)=1$.  
	\begin{center}
		\begin{tabular}{|M{5cm}|M{8cm}|}
			\hline
			Values of $p$  & Values of $C(\Gamma)$\\ 
			\hline
			1          &  2833.2070524   \\ \hline
			7              & 179826.7366816    \\ \hline
			11        &  -792274.9699035  \\ \hline
			13              & 1807438.9684806  \\ \hline
			17              & 5745222.85120843  \\ \hline
			19            &  -3699419.76217178 \\ \hline
			23            & 17634781.2649327  \\ \hline
			29           &  -20336547.9959923   \\ \hline
			31          &  -27809085.4970539  \\ \hline
			37          &  120775875.7114752  \\ \hline
			41          &  -84581711.0963804  \\ \hline
			43          &  216829661.9498101  \\ \hline
		\end{tabular}
	\end{center}
	
	\subsection{$\Gamma= P'_{8.3^k} \times \mathbb{Z}/p$:}
	Here we give the table of $C(\Gamma)$ for $1 \leq \tilde{p} \leq 23$, where $\tilde{p}=p \cdot 3^{m-1}$ and $gcd(p,8.3^m)=1$. 
	
	\begin{center}
		\begin{tabular}{|M{2.5cm}|M{5cm}|M{8cm}|}
			\hline
			Values of $m$ & Values of $p$ & Values of $C(\Gamma)$\\ 
			\hline
			1  & 1                & -484.8486298220    \\ \hline
			1  & 5                & 5030.1210024879     \\ \hline
			1  & 7               &  3676.7952801659  \\ \hline
			1  & 11               & 31610.0383198645  \\ \hline
			1  & 13              &  63761.5753483247 \\ \hline
			1  & 17              &  185455.4886479516  \\ \hline
			1  & 19               &  289290.8713510202  \\ \hline
			1  & 23               & 621453.3057542715 \\ \hline
			2  & 1               &  -1097.0507252631  \\ \hline
			2  & 5               &  112137.4453024487 \\ \hline
			2  & 7               &  431848.4432457592 \\ \hline
			3  & 1               &  16148.9146758443 \\ \hline
		\end{tabular}
	\end{center}
	
	\bibliography{biblioscv}
	\bibliographystyle{amsalpha.bst}
\end{document}